\documentclass[11pt,thmsa]{article}%
\usepackage{amsmath}
\usepackage{amsfonts}
\usepackage{amssymb}
\usepackage{graphicx}%
\usepackage{multirow}
\usepackage{chngpage}
\usepackage{array}
\usepackage{float}
\usepackage{booktabs}
\newtheorem{theorem}{Theorem}[section]

\newtheorem{corollary}[theorem]{Corollary}

\newtheorem{definition}[theorem]{Definition}

\newtheorem{lemma}[theorem]{Lemma}

\newenvironment{proof}[1][Proof]{\noindent\textbf{#1.} }{\ \rule{0.5em}{0.5em}}

\begin{document}
\title{Standard homogeneous $(\alpha_1,\alpha_2)$-metrics
\\and geodesic orbit property\footnote{Supported by National Natural Science Foundation of China (No. 11821101, No. 11771331), Beijing Natural Science Foundation
(No. 00719210010001, No. 1182006), Research Cooperation Contract (No. Z180004), Capacity Building for Sci-Tech  Innovation - Fundamental Scientific Research Funds (No. KM201910028021)}}
\author{Lei Zhang and Ming Xu\thanks{Corresponding author. Email:
mgmgmgxu@163.com.}
\\
\\
School of Mathematical Sciences and Information Science, Yantai University \\Yantai City 264005, P.R. China\\
\\
 E-mail: 546502871@qq.com \\
\\
School of Mathematical Sciences, Capital Normal University\\Beijing 100048 P.R. China \\
\\
E-mail:  mgmgmgxu@163.com}

\date{}
\maketitle
\begin{abstract}
In this paper, we introduce the notion of standard homogeneous $(\alpha_1,\alpha_2)$-metrics, as a natural non-Riemannian deformation for the normal homogeneous Riemannian metrics. We prove that with respect to the given bi-invariant inner product and orthogonal decompositions for
$\mathfrak{g}$, if there exists one generic standard g.o. $(\alpha_1,\alpha_2)$-metric, then all other standard homogeneous $(\alpha_1,\alpha_2)$-metrics are also g.o.. For standard homogeneous $(\alpha_1,\alpha_2)$-metrics associated with a triple of compact connected Lie groups, we can refine
our theorem and get some simple algebraic equations as the criterion for the g.o. property. As the application of this criterion, we discuss
standard g.o. $(\alpha_1,\alpha_2)$-metric from H. Tamaru's classication work, and find some new examples of non-Riemannian
g.o. Finsler spaces which are not weakly symmetric. On the other hand, we also prove that all standard g.o. $(\alpha_1,\alpha_2)$-metrics on the three Wallach spaces, $W^6=SU(3)/T^2$, $W^{12}=Sp(3)/Sp(1)^3$ and $W^{24}=F_4/Spin(8)$, must be the normal homogeneous Riemannian metrics.


\end{abstract}

\section{Introduction}

A homogeneous Riemannian manifold is called a geodesic orbit (or g.o. in short) space, if any geodesic is the orbit of a one-parameter subgroup of isometries. This notion was  introduced by O. Kowalski and L.Vanhecke in 1991 \cite{KV},  as a generalization of the naturally reductive homogeneity. Since then, there have been many research works on this subject. See 
\cite{GN} and the references therein for some recent progress.

Meanwhile, g.o. metrics have also been studied in Finsler geometry. In \cite{YD}, the notion of g.o. Finsler space was defined, with a similar style as in Riemannian geometry.
It is well known that weakly symmetric Finsler spaces, which include all globally symmetric Finsler spaces, are g.o.
\cite{D1}.


To construct more g.o. Finsler metrics on a coset space $G/H$ in which both $G$ and $H$ are compact connected Lie groups, we may start with the best homogeneous metric
we can find,
i.e., the $G$-normal homogeneous Riemannian metric $F$, which can be presented by submersion from a bi-invariant Riemannian metric $\bar{F}$ on $G$. Notice that $\bar{F}$ corresponds to a bi-invariant inner product $|\cdot|_{\mathrm{bi}}^2=\langle\cdot,\cdot\rangle_{\mathrm{bi}}$ on $\mathfrak{g}=\mathrm{Lie}(G)$.
Then we consider its non-Riemannian deformations.

There are two reasonable deformations
which may produce g.o. metric. The first is to keep the submersion
but change $\bar{F}$ among bi-invariant metrics on $G$.
Notice that $\bar{F}$ can even be singular. By this deformation, we define the notions of normal homogeneous Finsler spaces
\cite{XD} and generalized normal (or $\delta$-normal) homogeneous Finsler spaces \cite{XZ}. These two classes of homogeneous Finsler spaces have many interesting geometric properties. For example, they are g.o. Finsler spaces with non-negative flag curvature and vanishing S-curvature. However, they also have the the shortcoming of less computability.

In this paper, we concern the computability and
introduce
the second deformation for the normal homogeneous Riemannian metric on $G/H$ as following. Let $\mathfrak{g}=\mathfrak{h}+\mathfrak{m}$ and $\mathfrak{m}=\mathfrak{m}_1+\mathfrak{m}_2$ be orthogonal
decompositions with respect to the given bi-invariant inner product $\langle\cdot,\cdot\rangle_{\mathrm{bi}}$ on $\mathfrak{g}$, such that each summand is $\mathrm{Ad}(H)$-invariant. In Riemannian geometry,
$G$-invariant metric defined
by the following inner product
\begin{equation}\label{0040}
  \langle\cdot,\cdot\rangle_{a,b}
=a\langle\cdot,\cdot\rangle_{\mathrm{bi}}
|_{\mathfrak{m}_1\times\mathfrak{m}_1}\oplus b
\langle\cdot,\cdot\rangle_{\mathrm{bi}}|_{\mathfrak{m}_2
\times\mathfrak{m}_2}
\end{equation}
and its g.o. property has been intensively studied \cite{CN2017}.
In Finsler geometry, more metrics can be defined from $\mathrm{Ad}(H)$-invariant $(\alpha_1,\alpha_2)$-norms on $\mathfrak{m}$  generalizing those in (\ref{0040}),
i.e., $F(v)=|v|_{\mathrm{bi}}/
\varphi(|v_2|_{\mathrm{bi}}/|v|_{\mathrm{bi}})$, in which $\varphi(s)$ is a positive smooth function satisfying certain regularity requirements \cite{DX2016}\cite{HM2018}. This deformation defines a special class of homogeneous $(\alpha_1,\alpha_2)$-metrics, which we will simply call {\it standard homogeneous $(\alpha_1,\alpha_2)$-metrics}.
They reduce to the homogeneous Riemannian metrics defined
by $\langle\cdot,\cdot\rangle_{a,b}$ when $\varphi(s)=\sqrt{a+(b-a)s^2}$ and a normal homogeneous Riemannian metric when $\varphi(s)\equiv\mathrm{const}>0$.

Standard homogeneous $(\alpha_1,\alpha_2)$-metrics are much easier to be explicitly presented and more calculable, than most non-Riemannian normal homogeneous
or $\delta$-homogeneous Finsler metrics. But they may lose the g.o. property; see \cite{AW} and
Theorem \ref{last-thm} in Section 7, for examples in Riemannian and Finsler geometry respectively.
Meanwhile, we see an interesting phenomenon. Roughly speaking, with respect to the given bi-invariant inner product
and orthogonal decompositions, if there exists a generic standard $G$-g.o.
$(\alpha_1,\alpha_2)$-metric on $G/H$,
then all other standard homogeneous $(\alpha_1,\alpha_2)$-metrics on $G/H$ are $G$-g.o. as well. See Theorem \ref{Theorem 2} and Theorem
\ref{Theorem 3} for the precise statements.

As the application, we discuss the compact coset spaces associated with some special triples.
In \cite{T}, H. Tamaru classified
all triples $(H,K,G)$ such that $G/H$ is an effective irreducible
symmetric space, and for any $v_F\in\mathfrak{m}_F$ and
$v_B\in\mathfrak{m}_B$ with respect to the bi-invariant orthogonal
decompositions $\mathfrak{g}=\mathfrak{k}+\mathfrak{m}_B
=\mathfrak{h}+\mathfrak{m}_F+\mathfrak{m}_B$, there exists
a vector $u\in\mathfrak{h}$, such that
\begin{equation}\label{0051}
[u,v_F]=[u+v_F,v_B]=0.
\end{equation}
See Table \ref{TableI} or that in
\cite{T} for this classification in the Lie algebra level.
By Theorem \ref{Theorem 3}, all standard
homogeneous $(\alpha_1,\alpha_2)$-metrics associated with triples in Table \ref{TableI} are g.o. (see Corollary \ref{last-cor}). In particular, we find some new examples of g.o. Finsler spaces which are not weakly symmetric.

Finally,
we discuss the standard homogeneous $(\alpha_1,\alpha_2)$-metrics on the three Wallach spaces,
$W^6=SU(3)/T^2$, $W^{12}=Sp(3)/Sp(1)^3$ and $W^{24}=F_4/Spin(8)$, and prove that they are g.o. only when they are normal homogeneous Riemannian metrics (see Theorem \ref{last-thm}). So the second deformation proposed
in this paper is
not always effective for producing new g.o. metrics in Finsler geometry.

This paper is organized as following. In Section 2, we summarize some basic
knowledge in general and homogeneous
Finsler geometry and introduce the notion of standard homogeneous $(\alpha_1,
\alpha_2)$-metric. In Section 3, we recall the notion of the g.o. property in Finsler geometry and prove a criterion for g.o. $(\alpha_1,\alpha_2)$-spaces. In Section 4, we prove Theorem \ref{Theorem 2}, verifying our observation for the g.o. property of standard $(\alpha_1,\alpha_2)$-metrics. In
Section 5, we prove Theorem \ref{Theorem 3}, refining our observation for the g.o. property of standard $(\alpha_1,\alpha_2)$-metrics associated with triples. In Section 6,
we discuss the standard g.o. $(\alpha_1,\alpha_2)$-metrics on the
coset spaces in
H. Tamaru's classification, and the three Wallach spaces.

\section{Preliminaries}

\subsection{Minkowski norm and Finsler metric}

Firstly, we recall the notions of {\it Minkowski norm} and {\it Finsler metric} in \cite{BCS2000}. In this
paper, all manifolds are assumed to be connected and smooth.

\begin{definition}
A Minkowski norm on an $n$-dimensional real vector space $\mathbf{V}$ is a real continuous function $F:\mathbf{V}\rightarrow[0,\infty)$ on $\mathbf{V}$ satisfying the following conditions:
\begin{description}
  \item{\rm (1)} $F$ is positive and smooth when restricted to $\mathbf{V}\backslash \{0\}$;
  \item{\rm (2)} $F$ is positively homogeneous of degree one, i.e. $F(\lambda u)=\lambda F(u)$, $\forall u\in\mathbf{V}$ and $\lambda > 0$;
  \item{\rm (3)} $F$ is strongly convex. Namely, choose any basis $\{e_1,e_2,\cdots,e_n\}$ of $\mathbf{V}$ and write
  $F(y)=F(y^1,y^2,\cdots,y^n)$ for $y=y^ie_i\in\mathbf{V}$, then the Hessian matrix
  \begin{equation*}
    (g_{ij}(y))=([\frac{1}{2}F^2]_{y^iy^j})
  \end{equation*}
  is positive definite when $y\neq0$.
\end{description}
\end{definition}

\begin{definition}
A Finsler metric on a manifold $M$ is a continuous function $F: TM \rightarrow [0,\infty)$ such that
\begin{description}
  \item{\rm (1)} $F$ is smooth on the slit tangent bundle $TM\backslash 0$;
  \item{\rm (2)} The restriction of $F$ to any tangent space $T_x M$, $x\in M$, is a Minkowski norm.
\end{description}
\end{definition}

When the manifold $M$ is endowed with a Finsler metric, we will also call $(M,F)$ a {\it Finsler manifold} or a {\it Finsler space}. The Hessian matrix defines an inner product
$$g_y(u,v)=\frac12\frac{\partial^2}{\partial s\partial t}F^2(y+su+tv)|_{s=t=0},\forall u,v\in T_xM,$$
which is called the {\it fundamental tensor}.
We say $F$ is {\it reversible},
if $F(x,y)=F(x,-y)$ for any $x\in M$
and $y\in T_xM$.

The metric $F$ is Riemannian when $g_y(\cdot,\cdot)$ is irrelevant to the choice of $y$. The most important and simplest (non-Riemannian) Finsler metric is Randers metric, which is the form $F=\alpha+\beta$, in which $\alpha$ is a Riemannian metric
and $\beta$ is a one-form. Its generalization, the $(\alpha,\beta)$-metric, has the form of $F=\alpha\varphi(\beta/\alpha)$ with a positive smooth function $\varphi(s)$ and similar $\alpha$ and $\beta$ as for a Randers space. In \cite{DX2016}, the authors defines $(\alpha_1,\alpha_2)$-metric with a similar presenting form as an analog of a $(\alpha,\beta)$-metric, which will be recalled in the next subsection.

\subsection{Finsler metric and Minkowski norm of $(\alpha_1,\alpha_2)$-type}

Let $\alpha$ be a Riemannian metric on a manifold $M$. Suppose that we have a $\alpha$-orthogonal decomposition $TM=\mathcal{V}_1\oplus \mathcal{V}_2$, where $n_i>0$, $i=1$ or $2$, is the real dimension of the fibers of linear sub-bundle $\mathcal{V}_i$ respectively. Denote $\mathrm{pr}_i$ the
fiber-wise projection mapping each tangent vector to its component in $\mathcal{V}_i$, and $\alpha_i=\alpha\circ\mathrm{pr}_i$. Then
a Finsler metric which only depends on $\alpha_1$- and $\alpha_2$-values is called an {\it $(\alpha_1,\alpha_2)$-metric}.
To be precise, we have the following definition \cite{DX2016}.
\begin{definition}
With respect to a Riemannian metric $\alpha$ on the manifold $M$, an $\alpha$-orthogonal decomposition $TM=\mathcal{V}_1\oplus\mathcal{V}_2$ and $\alpha_i=\alpha|_{\mathcal{V}_i}$,
a Finsler metric $F$ is called an $(\alpha_1,\alpha_2)$-metric,
if it can be presented as $$F(x,y)=f(\alpha(x,y_1),\alpha(x,y_2)),\quad\forall y=y_1+y_2\in T_xM \mbox{ with }
 y_i\in\mathcal{V}_i, i=1,2,$$
for some positive smooth function $f(\cdot,\cdot)$.
\end{definition}

The pair $(n_1,n_2)$ is called the {\it dimension decomposition} for the $(\alpha_1,\alpha_2)$-metric. Slightly different from the setting in \cite{XD},
we permit $\min\{n_1,n_2\}=1$ here, which corresponds to all the reversible $(\alpha,\beta)$-metrics. An $(\alpha_1,\alpha_2)$-metric $F$ is Riemannian iff it is defined by $f(u,v)=\sqrt{au^2+bv^2}$ for some positive constants $a$ and $b$.

The restriction of an $(\alpha_1,\alpha_2)$-metric to each tangent space is called an {\it $(\alpha_1,\alpha_2)$-norm}. This notion
may also be abstractly defined for any real vector space \cite{DX2016}.

In later discussion, we will apply an alternative presenting $F=\alpha\varphi(\alpha_2/\alpha)$ for a
$(\alpha_1,\alpha_2)$-metric (or a $(\alpha_1,\alpha_2)$-norm), where $\varphi(s)$ is positive and smooth for $s\in(0,1)$. More dedicate discussion for the regularity of $F$ can even show the smoothness of $\varphi(s)$ and $s=0$ and $1$, and the inequalities
in the following lemma are satisfied.
\begin{lemma}\label{regularity-lemma-a1-a2}
If $F=\alpha\varphi(\alpha_2/\alpha)$ is a regular $(\alpha_1,\alpha_2)$-metric or $(\alpha_1,\alpha_2)$-norm,
then $\varphi(s)$ satisfies
$$\varphi(s)-s\varphi'(s)>0\mbox{ and }\varphi(s)-(s-s^{-1})\varphi'(s)>0$$
for any $s\in(0,1)$.
\end{lemma}
\begin{proof}
The inequality $\varphi(s)-s\varphi'(s)>0$ follows immediately
Theorem 3.2 in \cite{DX2016} or \cite{CS2005} for the $(\alpha,\beta)$-metrics. It can also be deduced by (21) in \cite{HM2018}. Denote $\psi(s)=\varphi(\sqrt{1-s^2})$, which can
be used to present $F$ as $F=\alpha\psi(\alpha_1/\alpha)$. So
$\varphi(s)-(s-s^{-1})\varphi'(s)=\psi(t)-t\psi'(t)$ must also be
positive for the same reason as above. This proves the other inequality.
\end{proof}

\subsection{Standard homogeneous $(\alpha_1,\alpha_2)$-space}

Firstly, we introduce basic setup in homogeneous Finsler geometry \cite{D1}.

Let $(G/H,F)$ be a homogeneous Finsler space such that isotropy subgroup $H$ is compactly imbedded and connected. We can find a linear
decomposition $\mathfrak{g}=\mathfrak{h}+\mathfrak{m}$,
where $\mathfrak{g}=\mathrm{Lie}(G)$ and $\mathfrak{h}=\mathrm{Lie}(H)$, such that $\mathrm{Ad}(H)\mathfrak{m}=\mathfrak{m}$ and thus
$[\mathfrak{h},\mathfrak{m}]\subset\mathfrak{m}$.
We usually call it
a {\it reductive decomposition} for $G/H$.
The subspace $\mathfrak{m}$ can be identified with $T_o(G/H)$
at $o=eH$, such that the $\mathrm{Ad}(H)$-action on $\mathfrak{m}$ coincides with the isotropy action.
The $G$-invariant Finsler metric $F$ is one-to-one determined
by its restriction to $T_o(G/H)$, i.e. an $\mathrm{Ad}(H)$-invariant Minkowski norm on $\mathfrak{m}$.

For example,
when we have chosen an $\mathrm{Ad}(H)$-invariant inner product $|\cdot|^2=\langle\cdot,\cdot\rangle$ on $\mathfrak{m}$, and a $\langle\cdot,\cdot\rangle$-orthogonal
$\mathrm{Ad}(H)$-invariant linear decomposition
$\mathfrak{m}=\mathfrak{m}_1+\mathfrak{m}_2$, with $n_i=\dim\mathfrak{m}_i>0$ for $i=1$ and $2$ respectively. Then the
$(\alpha_1,\alpha_2)$-norm
\begin{equation}\label{0014}
F(v)=|v|\varphi(|v_2|/|v|),
\quad\forall v=v_1+v_2\in\mathfrak{m}\backslash\{0\},\, v_i\in\mathfrak{m}_i,\,i=1,2,
\end{equation}
defines a homogeneous Finsler metric on $G/H$, which is still denoted as $F$ for simplicity.

Since $\mathfrak{m}_1$ and $\mathfrak{m}_2$ are $\mathrm{Ad}(H)$-invariant,
they define two linear sub-bundles $\mathcal{V}_1$ and
$\mathcal{V}_2$ of $T(G/H)$ which is
orthogonal at each point, with respect to the homogeneous Riemannian metric $\alpha$ defined by $\langle\cdot,\cdot\rangle$.
So the $G$-invariant metric $F$ defined by (\ref{0014})
can also be presented as $F=\alpha\varphi(\alpha_2/\alpha)$,
which is a homogeneous $(\alpha_1,\alpha_2)$-metric.

Then we add the assumption that $G$ is compact. We fix an
$\mathrm{Ad}(G)$-invariant (i.e. bi-invariant) inner product $|\cdot|_{\mathrm{bi}}^2=
\langle\cdot,\cdot\rangle_{\mathrm{bi}}$ on $\mathfrak{g}=\mathrm{Lie}(G)$.
Then we have the corresponding $\langle\cdot,\cdot\rangle_{\mathrm{bi}}$-orthogonal reductive decomposition
$
  \mathfrak{g}=\mathfrak{h}+\mathfrak{m}
$.

Assume that we have a non-trivial $\mathrm{Ad}(H)$-invariant bi-invariant
orthogonal decomposition
$\mathfrak{m}=\mathfrak{m}_1+\mathfrak{m}_2$.
The normal homogeneous Riemannian metric on $G/H$ defined by $\langle\cdot,\cdot\rangle_{\mathrm{bi}}
|_{\mathfrak{m}\times\mathfrak{m}}$,
has a canonical two parameter family of deformations in homogeneous Riemannian geometry, i.e. for each pair of positive numbers $a$ and $b$, there is a unique homogeneous Riemannian metric on $G/H$ defined by the inner product
\begin{equation}\label{inner-product-a-b}
\langle\cdot,\cdot\rangle_{a,b}=
a\langle\cdot,\cdot\rangle
|_{\mathfrak{m}_1\times\mathfrak{m}_1}
\oplus b\langle\cdot,\cdot\rangle
|_{\mathfrak{m}_2\times\mathfrak{m}_2}.
\end{equation}

When the isotropy representation is the sum of two inequivalent
irreducible summands, the decomposition $\mathfrak{m}=\mathfrak{m}_1+\mathfrak{m}_2$ mentioned above can be uniquely determined.
Furthermore, in this case any homogeneous Riemannian metric on $G/H$ is induced by some $\langle\cdot,\cdot\rangle_{a,b}$ in (\ref{inner-product-a-b}).
See \cite{DK} for the complete classification for the compact coset space $G/H$ in this situation, up to finite coverings and automorphisms of $G$.

The homogeneous Riemannian deformation  $\langle\cdot,\cdot\rangle_{a,b}$ can be generalized to Finsler geometry as following.

\begin{definition}
Let $(G/H,F)$ be a homogeneous Finsler space on which the compact connected Lie group $G$ acts transitively. We call $F$ a standard homogeneous $(\alpha_1,\alpha_2)$-metric, or
call $(G/H,F)$ a standard homogeneous $(\alpha_1,\alpha_2)$-space, for the bi-invariant inner product $|\cdot|_{\mathrm{bi}}^2=
\langle\cdot,\cdot\rangle_{\mathrm{bi}}$ on $\mathfrak{g}$
and the $\mathrm{Ad}(H)$-invariant $\langle\cdot,\cdot\rangle_{\mathrm{bi}}$-orthogonal linear decompositions
$\mathfrak{g}=\mathfrak{h}+\mathfrak{m}$ and $\mathfrak{m}=\mathfrak{m}_1+\mathfrak{m}_2$,
if the Minkowski norm that $F$ defines on $\mathfrak{m}$ has the form of (\ref{0014}), i.e. $F(v)=|v|_{\mathrm{bi}}\varphi(|v_2|_{\mathrm{bi}}
/|v|_{\mathrm{bi}})$, for any nonzero vector $v=v_1+v_2\in\mathfrak{m}$ with $v_i\in\mathfrak{m}_i$ respectively.
\end{definition}

If we present a standard homogeneous $(\alpha_1,\alpha_2)$-metric
as $F=\alpha\varphi(\alpha_2/\alpha)$, $\alpha$ is then the normal homogeneous Riemannian metric defined by the given bi-invariant
inner product. The space of all standard homogeneous $(\alpha_1,\alpha_2)$-metrics on $G/H$ for the given bi-invariant inner product and decomposition is an infinite dimensional manifold parametrized by the positive smooth function $\varphi(s)$ on $[0,1]$ satisfying some regularity conditions. They include
the Riemannian ones defined by $\langle\cdot,\cdot\rangle_{a,b}$,
when $\varphi(s)=\sqrt{a+(b-a)s^2}$.

\section{Geodesic orbit $(\alpha_1,\alpha_2)$-spaces}

We recall the definition of {\it geodesic orbit Finsler space} in \cite{YD}.

\begin{definition}\label{definition 3.1}
Let $(M,F)$ be a connected Finsler space on which
a connected Lie group $G$ acts effectively and isometrically. We call a geodesic $c(t)$ on $(M,F)$ $G$-homogeneous, if it is
the orbit of a one-parameter subgroup in $G$, i.e. $c(t)=\exp tX\cdot x$ for some nonzero vector
$X\in\mathfrak{g}=\mathrm{Lie}(G)$. We call $(M,F)$
a $G$-geodesic orbit (or $G$-g.o.) Finsler space, if
any geodesic of positive constant speed is $G$-homogeneous.
\end{definition}

The g.o. Finsler spaces include many important subclasses,  normal homogeneous Finsler spaces \cite{XD}, $\delta$-homogeneous
Finsler spaces \cite{XZ}, Clifford--Wolf homogeneous Finsler spaces \cite{XD2013}, etc..
In particular, weakly symmetric
Finsler spaces are g.o. \cite{DH2013}.
For example, any $Spin(8)$-invariant Finsler metric on $Spin(8)/G_2$ is weakly symmetric
(see Proposition 6.2 in \cite{D1}), thus it is g.o..
More details about $Spin(8)/G_2$ can be found  in \cite{Kerr}\cite{Ziller}.

The nonzero vector $X\in\mathfrak{g}$ in Definition \ref{definition 3.1} is called a {\it geodesic vector}.
Given a reductive decomposition $\mathfrak{g}=\mathfrak{h}+\mathfrak{m}$, geodesic vectors
can be equivalently described by the following lemma in
\cite{YD}.

\begin{lemma}\label{criterion-geodesic-vector}
Let $(G/H,F)$ be a homogeneous Finsler space on which the connected Lie group $G$ acts effectively, and $\mathfrak{g}
=\mathfrak{h}+\mathfrak{m}$ is a reductive decomposition.
Then $X\in\mathfrak{g}$ is a $G$-geodesic vector iff
its $\mathfrak{m}$-factor $X_\mathfrak{m}$ is not zero, and
\begin{equation}\label{0015}
  g_{X_{\mathfrak{m}}}(X_{\mathfrak{m}},
  [X,\mathfrak{m}]_{\mathfrak{m}})=0.
\end{equation}
\end{lemma}


By Definition \ref{definition 3.1},
the homogeneous Finsler space $(G/H, F)$ with the given reductive decomposition $\mathfrak{g}=\mathfrak{h}+\mathfrak{m}$ is $G$-g.o. iff $\mathrm{pr}_\mathfrak{m}$ maps the subset of
all $G$-geodesic vectors onto $\mathfrak{m}\backslash\{0\}$.

When $(G/H,F)$ is a homogeneous $(\alpha_1,\alpha_2)$-space,
we can apply Lemma \ref{criterion-geodesic-vector} to prove the following
criterion theorem.

\begin{theorem}\label{Theorem 1}
Let $F=\alpha\varphi(\alpha_2/\alpha)$ be a homogeneous $(\alpha_1,\alpha_2)$-metric on $G/H$ with respect to
the reductive decomposition $\mathfrak{g}=\mathfrak{h}+\mathfrak{m}$, the
$\mathrm{Ad}(H)$-invariant inner product $|\cdot|^2=\langle\cdot,\cdot\rangle$ and the $\mathrm{Ad}(H)$-invariant
$\langle\cdot,\cdot\rangle$-orthogonal decompositions
$\mathfrak{m}=\mathfrak{m}_1+\mathfrak{m}_2$.
Then $X\in \mathfrak{g}$
is a $G$-geodesic vector if and only if one of the following is satisfied:
\begin{description}
\item{\rm (1)} When $X_\mathfrak{m}
    \in (\mathfrak{m}_1\cup\mathfrak{m}_2)\backslash\{0\}$,
    we have \begin{equation}\label{0003}
  \langle [X,Z]_{\mathfrak{m}},X_{\mathfrak{m}} \rangle=0,\  \forall Z\in \mathfrak{m}.
\end{equation}
\item{\rm (2)} When
$X_\mathfrak{m}\in\mathfrak{m}\backslash
(\mathfrak{m}_1\cup\mathfrak{m}_2)$, we have
\begin{equation}\label{0002}
  \langle [X,\mathfrak{m}]_\mathfrak{m}, (\varphi(\theta)-\theta\varphi'(\theta))X_1+ (\varphi(\theta)-(\theta-\theta^{-1})\varphi'(\theta)) X_2\rangle=0.
\end{equation}
\end{description}
Here
$X=X_\mathfrak{h}+X_\mathfrak{m}=X_\mathfrak{h}+X_{1}+X_{2}$ is according to the decomposition $\mathfrak{g}=\mathfrak{h}+\mathfrak{m}=
\mathfrak{h}+\mathfrak{m}_1+
\mathfrak{m}_2$,
and $\theta=|X_2|/|X|\in[0,1]$.
\end{theorem}

\begin{proof}
For any $U\in\mathfrak{m}\backslash\{0\}$ and $V\in\mathfrak{m}$, denote
$U=U_1+U_2$ and $V=V_1+V_2$ with $U_i,V_i\in\mathfrak{m}_i$ respectively. The fundamental tensor satisfies
\begin{eqnarray*}
g_U(U,V)=
\frac12\frac{\partial}{\partial t}F^2(U+tV)|_{t=0},
\end{eqnarray*}
where $$F^2(U+tV)=|U+tV|^2\varphi^2
(|U_2+tV_2|/|U+tV|).$$
Direct calculation shows, when
$\theta=|U_2|/|U|\in(0,1)$,
\begin{eqnarray}\label{0001}
g_U(U,V)=\varphi(\theta)\langle(\varphi(\theta)-
\theta\varphi'(\theta))U_1+(\varphi(\theta)-(\theta-\theta^{-1})
\varphi'(\theta))U_2,V\rangle.
\end{eqnarray}

When $X_\mathfrak{m}$ is contained in $\mathfrak{m}\backslash
(\mathfrak{m}_1\cup\mathfrak{m}_2)$, we can
take $U=X_\mathfrak{m}$ and $V=[X,Z]_\mathfrak{m}$ for any
$Z\in\mathfrak{m}$. By Lemma \ref{criterion-geodesic-vector} and
(\ref{0001}), we see (\ref{0002}) is the equivalent condition for $X$ to be a $G$-geodesic vector for $F$ in this case.

When the nonzero vector $X_\mathfrak{m}$ is contained in $\mathfrak{m}_1$ or
$\mathfrak{m}_2$, its $g_{X_\mathfrak{m}}$-orthogonal
complement coincides with its $\langle\cdot,\cdot\rangle$-orthogonal complement
in $\mathfrak{m}$. So in this case, by Lemma \ref{criterion-geodesic-vector}, $X$ is a $G$-geodesic vector for
$F$ iff it is a $G$-geodesic vector
for the homogeneous Riemannian metric $\alpha$ defined by $\langle\cdot,\cdot\rangle$. Using Lemma \ref{criterion-geodesic-vector} again for $\alpha$, we see
that (\ref{0003}) is an equivalent condition for $X$ to be
a $G$-geodesic vector for $F$ in this case.

This ends the proof of this theorem.
\end{proof}

\section{ Standard g.o. $(\alpha_1,\alpha_2)$-metrics}

Let $G/H$ be a smooth coset space on which the compact connected Lie group $G$ acts effectively. We choose an $\mathrm{Ad}(G)$-invariant inner product $\langle\cdot,\cdot\rangle_{\mathrm{bi}}$ and an
$\mathrm{Ad}(H)$-invariant $\langle\cdot,\cdot\rangle_{\mathrm{bi}}$-orthogonal decomposition
$\mathfrak{g}=\mathfrak{h}+\mathfrak{m}=\mathfrak{h}+\mathfrak{m}_1
+\mathfrak{m}_2$.

Then we consider a standard homogeneous
$(\alpha_1,\alpha_2)$-metric $F=\alpha\varphi(\alpha_2/\alpha)$
for the given bi-invariant inner product and decomposition,
where $\varphi(s)$ is a non-constant positive smooth function on $[0,1]$
and $\alpha$ is the normal homogeneous Riemannian metric defined
by $\langle\cdot,\cdot\rangle_{\mathrm{bi}}$.

\begin{theorem} \label{Theorem 2}
Let $G/H$ be a smooth coset space on which the compact connected Lie group $G$ acts effectively. Choose
a given bi-invariant inner product
$\langle\cdot,\cdot\rangle_{\mathrm{bi}}$ on $\mathfrak{g}$,
and an $\mathrm{Ad}(H)$-invariant $\langle\cdot,\cdot\rangle_{\mathrm{bi}}$-orthogonal
decomposition $\mathfrak{g}=\mathfrak{h}+\mathfrak{m}_1+\mathfrak{m}_2$, with
both $\dim\mathfrak{m}_i>0$.
Then the followings are equivalent:
\begin{description}
\item{\rm (1)} There exists a standard homogeneous $G$-g.o. $(\alpha_1,\alpha_2)$-metric $F=\alpha\varphi(\alpha_2/\alpha)$ for the given bi-invariant inner product and decomposition, such that
$\varphi'(s)\neq0$ for any $s\in(0,1)$.
\item{\rm (2)} Any standard g.o. $(\alpha_1,\alpha_2)$-metric
$F$ for the given bi-invariant inner product and decomposition
is $G$-g.o..
\item{\rm (3)} For any nonzero vectors $v_i\in\mathfrak{m}_i$ and  positive numbers $c_i\in\mathbb{R}$ for $i=1$ and $2$ respectively,
we can find a vector $u\in\mathfrak{h}$ such that
\begin{equation}\label{0020}
[u,c_1v_1+c_2v_2]+[v_1,v_2]_\mathfrak{m}=0.
\end{equation}
\end{description}
\end{theorem}

\begin{proof}
Firstly, the proof from (2) to (1) is obvious.

Nextly, we prove the statement from (3) to (2).
Let $F$ be any standard g.o. $(\alpha_1,\alpha_2)$-metric, which is defined
$F(v)=|v|_{\mathrm{bi}}\varphi
(|v_2|_{\mathrm{bi}}/|v|_{\mathrm{bi}})$,
in which $v=v_1+v_2$ is any nonzero vector in $\mathfrak{m}$ with
$v_i\in\mathfrak{m}_i$ respectively, and $\varphi(s)$ is a smooth function on $[0,1]$.

If $v\in\mathfrak{m}_1\cup\mathfrak{m}_2$, or
$v\in \mathfrak{m}\setminus (\mathfrak{m}_1\cup\mathfrak{m}_2)$ and
$\varphi'(\theta)=0$ in which $\theta=|v_2|_{\mathrm{bi}}
/|v|_{\mathrm{bi}}\in(0,1)$,
we may choose $u=0\in\mathfrak{h}$. Then by the bi-invariance
of $|\cdot|^2_{\mathrm{bi}}
=\langle\cdot,\cdot\rangle_{\mathrm{bi}}$, $X=v\in\mathfrak{g}$ satisfies (\ref{0003}) and (\ref{0002}) in Theorem \ref{Theorem 1}
respectively.

If $\theta\in(0,1)$ and $\varphi'(\theta)\neq0$, by Lemma \ref{regularity-lemma-a1-a2} and the statement (3) of this theorem, we can find $u\in\mathfrak{h}$
satisfying (\ref{0020}), in which
$$c_1=\frac{\theta\varphi(\theta)-\theta^2\varphi'(\theta)
}{|\varphi'(\theta)|}\quad\mbox{and}\quad
c_2=\frac{\theta\varphi(\theta)-(\theta^2-1)\varphi'(\theta)
}{|\varphi'(\theta)|},$$
i.e.,
$$\langle(\varphi(\theta)-\theta\varphi'(\theta))[\pm u,v_1]
+(\varphi(\theta)-(\theta-\theta^{-1})\varphi'(\theta))[\pm u,v_2]
+\theta^{-1}\varphi'(\theta)[v_1,v_2]_\mathfrak{m},
\mathfrak{m}\rangle_{\mathrm{bi}}=0,$$
in which $\pm$'s are taken with respect to the sign of $\varphi'(\theta)$.
Then $X=\pm u+v\in\mathfrak{g}$ satisfies (\ref{0002}) in
Theorem \ref{Theorem 1}.

To summarize, by Theorem \ref{Theorem 1}, we have proved (2), i.e., the g.o. property of $F$.

Finally, we prove the statement from (1) to (3).

Let $F$ be the standard g.o. $(\alpha_1,\alpha_2)$-metric indicated in (1), which can be similarly determined as above by the norm
$F(v)=|v|_{\mathrm{bi}}\varphi
(|v_2|_{\mathrm{bi}}/|v|_{\mathrm{bi}})$, in which the smooth function $\varphi(s)$ has a nonzero derivative at any $s\in(0,1)$.
Let $v_i\in\mathfrak{m}_i$ be any nonzero vector and $c_i$ any positive number, for $i=1$ and $2$ respectively. For any positive number $t$, we denote $v_1(t)=tv_1$,
$v_2(t)=t^{-1}v_2$ and $\theta(t)=|v_2(t)|_{\mathrm{bi}}/|v_1(t)+
v_2(t)|_{\mathrm{bi}}\in(0,1)$.
By the
g.o. property of $F$, we can find $u'\in\mathfrak{h}$, such that
$X=u'+v_1(t)+v_2(t)\in\mathfrak{g}$ satisfies (\ref{0002}) in Theorem
\ref{Theorem 1}, i.e., $u''=\theta(t)\varphi'(\theta(t))^{-1}u'$
satisfies
\begin{equation}\label{0030}
c_1(t)[u'',v_1]
+c_2(t)\varphi'(\theta))[ u'',v_2]
+[v_1,v_2]_\mathfrak{m}
=0,\end{equation}
in which $c_1(t)=t(\varphi(\theta(t))-\theta(t)\varphi'(\theta(t)))$ and
$c_2(t)=t^{-1}(\varphi(\theta)-(\theta-\theta^{-1})\varphi'(\theta))$.
The regularity of $F$ requires that $\varphi(s)$ is smooth and $\varphi(s)-s\varphi'(s)>0$ for $s\in [0,1]$. So $\varphi(s)-s\varphi'(s)$ has positive lower and upper bounds
for $s\in(0,1)$. There are similar regularity requirement
for $\psi(s)=\varphi(\sqrt{1-s^2})$, so $\varphi(s)-(s-s^{-1})\varphi'(s)=(\psi(\tilde{s})-
\tilde{s}\psi'(\tilde{s}))|_{\tilde{s}=\sqrt{1-s^2}}$
also has positive lower and upper bounds. Then it is easy to
see 
\begin{eqnarray*}
\lim_{t\rightarrow0}c_1(t)=\lim_{t\rightarrow+\infty}c_2(t)=0\quad
\mbox{and}\quad
\lim_{t\rightarrow+\infty}c_1(t)=\lim_{t\rightarrow0}c_2(t)=+\infty.
\end{eqnarray*}
By the continuity, we can find positive numbers $t_0$ and $c'$, such that $c_1(t_0)/c_1=c_2(t_0)/c_2=c'$. Then $u=c'u''$ satisfies
(\ref{0020}).

To summarize, this argument proves the statement from (1) to (3),
and ends the proof of this theorem.
\end{proof}

For example, any standard homogeneous Riemannian $(\alpha_1,\alpha_2)$-metric defined by the inner product
$\langle\cdot,\cdot\rangle_{a,b}$ with $a\neq b$ satisfies the
requirement in (1) of Theorem \ref{Theorem 2}. So we have the following immediate
corollary.

\begin{corollary}
Let $G$ be a compact Lie group endowed with
a bi-invariant inner product $\langle\cdot,\cdot\rangle_{\mathrm{bi}}$ on its Lie algebra
$\mathfrak{g}$, and $H$ a compact subgroup of $G$.
Suppose that there exists a Riemannian g.o. metric on $G/H$ defined by $a\langle\cdot,\cdot\rangle_{\mathrm{bi}}|_{\mathfrak{m}_1}
\oplus b\langle\cdot,\cdot\rangle_{\mathrm{bi}}|_{\mathfrak{m}_2}$, where $a\neq b$ are positive numbers and $\mathfrak{g}=\mathfrak{h}+\mathfrak{m}_1+\mathfrak{m}_2$ is a
$\langle\cdot,\cdot\rangle_{\mathrm{bi}}$-orthogonal decomposition. Then any standard homogeneous $(\alpha_1,\alpha_2)$-metric
with respect the given bi-invariant inner product and decomposition is g.o..
\end{corollary}

\section{Standard g.o. $(\alpha_1,\alpha_2)$-metrics associated with a triple}

The requirement $\varphi'(s)\neq0$ in (1) of Theorem \ref{Theorem 2} restrict the  applications of Theorem \ref{Theorem 2}. In this section, we show that this restriction
can be weakened when we consider standard homogeneous $(\alpha_1,\alpha_2)$-metrics associated with
special triples.

Let $(H,K,G)$ be a triple of compact Lie groups, which Lie algebras are denoted as $\mathfrak{h}$, $\mathfrak{k}$ and
$\mathfrak{g}$ respectively, satisfying $H\subset K\subset G$
with $\dim H<\dim K<\dim G$.
With respect to a given bi-invariant inner product $|\cdot|_{\mathrm{bi}}^2=\langle\cdot,\cdot\rangle_{\mathrm{bi}}$
on $\mathfrak{g}$, we have the orthogonal decompositions
\begin{eqnarray*}
& &\mathfrak{g}=\mathfrak{k}+\mathfrak{m}_B,\phantom{\,\,\,\,\, \mbox{ and }\quad}
\mathfrak{k}=\mathfrak{h}+\mathfrak{m}_F,\\
& &\mathfrak{g}=\mathfrak{h}+\mathfrak{m}, \quad\mbox{ and } \quad
\mathfrak{m}=\mathfrak{m}_F+\mathfrak{m}_B.
\end{eqnarray*}
Here we denote $\mathfrak{m}_F$ and
$\mathfrak{m}_B$ the summands in $\mathfrak{m}$, and vectors $v_F$ and
$v_B$ in them accordingly, rather than $\mathfrak{m}_i$ and $v_i$ with $i=1$ and $2$ previously, to imply they are related to the base and fiber respectively, for the fiber bundle
$$F=K/H\rightarrow G/H\rightarrow M=G/K$$
respectively.

A standard homogeneous $(\alpha_1,\alpha_2)$-metric on $G/H$ with respect to
the bi-invariant inner product and decomposition described above
is called a {\it standard} $(\alpha_1,\alpha_2)$-metric on $G/H$ {\it associated
with the triple} $(H,K,G)$.

Besides that we have
$[\mathfrak{h},\mathfrak{m}_F]\subset\mathfrak{m}_F$ and
$[\mathfrak{h},\mathfrak{m}_B]\subset\mathfrak{m}_B$,
there is an extra bracket relation
$[\mathfrak{m}_F,\mathfrak{m}_B]\subset\mathfrak{m}_B$,
which helps to simply the criterion of geodesic vector and g.o. property to the following lemma.

\begin{lemma}\label{lemma3}
Let $(G/H,F)$ be a standard homogeneous $(\alpha_1,\alpha_2)$-space associated with the triple $(G,K,H)$ of compact connected Lie groups, and keep all relevant assumptions and notations. Then we have the following:
\begin{description}
\item{\rm (1)} Suppose that $v\in\mathfrak{m}$ is a nonzero vector and there exists a vector $u\in\mathfrak{h}$ satisfying
\begin{equation}\label{0005}
[u,v_F]=0\quad\mbox{and}\quad [u+v_F,v_B]=0,
\end{equation}
where $v_F$ and $v_B$ are the component of $v$ in $\mathfrak{m}_F$ and $\mathfrak{m}_B$ respectively. Then
we can find a $G$-geodesic vector for $(G/H,F)$ in $v+\mathfrak{h}$.
\item{\rm (2)} Suppose that for any nonzero vector $v\in\mathfrak{m}$ we can find a vector $u\in\mathfrak{h}$
    satisfying (\ref{0005}). Then $(G/H,F)$ is $G$-g.o..
\end{description}
\end{lemma}

\begin{proof}
We assume that the standard homogeneous $(\alpha_1,\alpha_2)$-metric $F$ is defined by
$F(v)=|v|_{\mathrm{bi}}\varphi
(|v_B|_{\mathrm{bi}}/|v|_{\mathrm{bi}})$, in which $v=v_F+v_B$ with $v_B\in\mathfrak{m}_B$ and $v_F\in\mathfrak{m}_F$.

We have seen in the proof of Theorem \ref{Theorem 2} that
when $v_F=0$ or $v_B=0$, $X=v\in\mathfrak{g}$ is a geodesic vector. Meanwhile, it is obvious to see that $u=0\in\mathfrak{h}$
satisfies (\ref{0005}) in this case. Hence we only need to consider the case $v_F\neq0$ and $v_B\neq0$ when proving (1) in
the lemma.

By (\ref{0002}) in Theorem \ref{Theorem 1}, for any
nonzero vectors $v_B\in\mathfrak{m}_B$ and $v_F\in\mathfrak{m}_F$,
there exists a geodesic vector $X=u'+v=u'+v_F+v_B\in v+\mathfrak{h}$ iff
\begin{eqnarray}\label{0011}
& &(\varphi(\theta)-\theta\varphi'(\theta)))
\langle [v_F,u'+v_B],Z\rangle_{\mathrm{bi}}
+(\varphi(\theta)-(\theta-\theta^{-1})\varphi'(\theta))
\langle Z,[v_B,u'+v_F]\rangle_{\mathrm{bi}}\nonumber\\
&=&\langle[u'+v_F+v_B,Z]_\mathfrak{m},
(\varphi(\theta)-\theta\varphi'(\theta))v_F
+(\varphi(\theta)-(\theta-\theta^{-1})\varphi'(\theta))v_B
\rangle_{\mathrm{bi}}\nonumber\\
&=&0,\quad\forall Z\in\mathfrak{m}_F\cup\mathfrak{m}_B,
\end{eqnarray}
in which $\theta=|v_B|_{\mathrm{bi}}/|v|_{\mathrm{bi}}\in(0,1)$.

Because $[v_B,v_F]$ and $[v_B,u']$ are vectors in $\mathfrak{m}_B$, (\ref{0011}) is valid for any $Z\in\mathfrak{m}_F$ if and only if $[u',v_F]=0$. After assuming $[u',v_F]=0$, (\ref{0011}) for any $Z\in\mathfrak{m}_B$ can be reduced to
\begin{equation}\label{0012}
\theta^{-1}\varphi'(\theta)
\langle Z,[v_B,v_F]\rangle_{\mathrm{bi}}
+(\varphi(\theta)-(\theta-\theta^{-1})\varphi'(\theta))
\langle Z,[v_B,u']\rangle_{\mathrm{bi}}=0,\quad\forall Z\in\mathfrak{m}_B.
\end{equation}

Since we have assumed the existence of $u\in\mathfrak{h}$
satisfying (\ref{0005}), the vector
\begin{equation}\label{0013}
u'=\frac{\theta\varphi'(\theta)}{
\theta\varphi(\theta)-(\theta^2-1)\varphi'(\theta)
} u\in\mathfrak{h}
\end{equation}
satisfies both $[u',v_F]=0$ and (\ref{0012}), i.e.,
$X=u'+v$ is a geodesic vector in $\mathfrak{g}$.

This ends the proof of (1) in the lemma. The statement (2) in the lemma follows (1) immediately.
\end{proof}

The main theorem of this section is the following.

\begin{theorem}\label{Theorem 3}
Let $(H,K,G)$ be a triple of compact connected Lie groups satisfying $H\subset K\subset G$ and $\dim H<\dim K<\dim G$.
Fix an $\mathrm{Ad}(G)$-invariant $|\cdot|_{\mathrm{bi}}^2=\langle\cdot,\cdot\rangle_{\mathrm{bi}}$, and
$\langle\cdot,\cdot\rangle_{\mathrm{bi}}$-orthogonal decompositions $\mathfrak{g}=\mathfrak{k}+\mathfrak{m}_B$,
$\mathfrak{k}=\mathfrak{h}+\mathfrak{m}_F$ and $\mathfrak{m}=\mathfrak{m}_F+\mathfrak{m}_B$. Then
we have the following
equivalent statements:
\begin{description}
\item{\rm (1)} There exists a $G$-g.o. standard homogeneous $(\alpha_1,\alpha_2)$-metric on $G/H$ associated with the triple $(G,K,H)$ and $\langle\cdot,\cdot\rangle_{\mathrm{bi}}$, which is not a $G$-normal homogeneous Riemannian metric;
\item{\rm (2)} All standard homogeneous $(\alpha_1,\alpha_2)$-metrics on $G/H$ associated with the triple $(G,K,H)$ and $\langle\cdot,\cdot\rangle_{\mathrm{bi}}$ are $G$-g.o.;
\item{\rm (3)} There exists a constant $\theta\in (0,1)$ such that for any vector $v=v_F+v_B\in\mathfrak{m}$ with $v_F\in\mathfrak{m}_F$,
    $v_B\in\mathfrak{m}_B$ and $|v_F|_{\mathrm{bi}}/|v_F+v_B|_{\mathrm{bi}}=\theta$, we
    can find a vector $u\in\mathfrak{h}$ satisfying
    $[u,v_F]=[u+v_F,v_B]=0$.
\item{\rm (4)} For any nonzero vector $v\in\mathfrak{m}$, we can find a vector $u\in\mathfrak{h}$ satisfying (\ref{0005}), i.e. $[u,v_F]=[u+v_F,v_B]=0$, where $v_F$ and $v_B$ are the
    components of $v$ in $\mathfrak{m}_F$ and $\mathfrak{m}_B$
    respectively.
\end{description}
\end{theorem}

\begin{proof}
Firstly, we prove the claim from (1) to (3).

Assume the standard homogeneous $(\alpha_1,\alpha_2)$-metric
indicated in (1) is defined by
$F(v)=|v|_{\mathrm{bi}}
\varphi(|v_B|_{\mathrm{bi}}/|v|_{\mathrm{bi}})$,
where $v=v_F+v_B\in\mathfrak{m}\backslash\{0\}$ with
$v_F\in\mathfrak{m}_F$ and $v_B\in\mathfrak{m}_B$. Because $F$ is not a $G$-normal homogeneous Riemannian metric, $\varphi(s)$ is
not a constant function. We may assume $\varphi'(\theta_0)\neq0$ for
some $\theta_0\in(0,1)$.

Let $v=v_F+v_B$ be any vector in $\mathfrak{m}$ satisfying
$|v_B|_{\mathrm{bi}}/|v|_{\mathrm{bi}}=\theta_0$.
Since $F$ is $G$-g.o., there exists a vector $u'\in\mathfrak{h}$
such that
$X=u'+v$ is a $G$-geodesic vector. By similar argument as for
Lemma \ref{lemma3}, we can see
$$u=\frac{
\theta_0\varphi(\theta_0)-(\theta_0^2-1)\varphi'(\theta_0)
}{\theta_0\varphi'(\theta_0)}u'$$
satisfies(\ref{0005}).

This argument proves the statement from (1) to (3).


Secondly, we prove the claim from (3) to (4).

Let $\theta_0\in(0,1)$ be the constant indicated in (3). If $v_F=0$ or
$v_B=0$, obviously $u=0$ satisfies (\ref{0005}), i.e., the requirement in (4) of the theorem.
When $v_F\neq0$ and $v_B\neq0$, we can find a suitable $\lambda>0$, such that $v'_F=v_F$ and $v'_B=\lambda v_B$ satisfy
$|v'_B|_{\mathrm{bi}}/|v'_F+v'_B|_{\mathrm{bi}}=\theta_0$. The
statement (2) provides the vector $u'\in\mathfrak{h}$ satisfying
$[u',v'_F]=[u'+v'_F,v'_B]=0$, which also satisfies (\ref{0005})
for $v=v_F+v_B$.

This argument proves the statement from (3) to (4).

Finally, the claim from (4) to (2) follows Lemma \ref{lemma3},
and the claim from (2) to (1) is obvious.
This ends the proof of Theorem \ref{Theorem 3}.
\end{proof}

\section{Examples of standard g.o. $(\alpha_1,\alpha_2)$-spaces}

In this section, we discuss some examples of
standard g.o. $(\alpha_1,\alpha_2)$-spaces associated with triples.
We keep all relevant notations for triples of compact Lie groups
in the last section.

Recall that in \cite{T} H. Tamaru has discussed the triples $(G,K,H)$ satisfying the following algebraic condition.

{\bf Condition I:} $(G,K)$ is a compact effective irreducible symmetric pair, and for every $v_F\in \mathfrak{m}_F$ and $v_B\in \mathfrak{m}_B$, there exists $u\in \mathfrak{h}$ satisfying (\ref{0005}), i.e.,
\begin{equation*}
  [u,v_F]=0\ \ \ \mbox{and}\ \ \ [u+v_F,v_B]=0.
\end{equation*}

He classified all these triples
in the Lie algebra level, by Table \ref{TableI}.

\begin{table}[htbp]
\centering
\caption{ The triples satisfying Condition I}\label{TableI}
\begin{tabular}{|c|c|c|c|c|}
  \hline
        & $\mathfrak{g}$          &$\mathfrak{k}$       &$\mathfrak{h}$     &    \\
  \hline
  (1.1) & $\mathfrak{so}(2n+1)$   &$\mathfrak{so}(2n)$  &$\mathfrak{u}(n)$  & $n\geq 2$  \\
  \hline
  (1.2) & $\mathfrak{so}(4n+1)$   &$\mathfrak{so}(4n)$  &$\mathfrak{su}(2n)$  & $n\geq 1$  \\
  \hline
  (1.3) & $\mathfrak{so}(8)$   &$\mathfrak{so}(7)$  &$\mathfrak{g}_2$        &   \\
  \hline
  (1.4) & $\mathfrak{so}(9)$   &$\mathfrak{so}(8)$  &$\mathfrak{so}(7)$       &   \\
  \hline
  (1.5) & $\mathfrak{su}(n+1)$   &$\mathfrak{u}(n)$  &$\mathfrak{su}(n)$  & $n\geq 2$  \\
  \hline
  (1.6) & $\mathfrak{su}(2n+1)$   &$\mathfrak{u}(2n)$  &$\mathfrak{u}(1)\oplus \mathfrak{sp}(n)$  & $n\geq 2$  \\
  \hline
  (1.7) & $\mathfrak{su}(2n+1)$   &$\mathfrak{u}(2n)$  &$\mathfrak{sp}(n)$  & $n\geq 2$  \\
  \hline
  (1.8) & $\mathfrak{sp}(n+1)$   &$\mathfrak{sp}(1)\oplus \mathfrak{sp}(n)$  &$\mathfrak{u}(1)\oplus \mathfrak{sp}(n)$  & $n\geq 1$  \\
  \hline
  (1.9) & $\mathfrak{sp}(n+1)$   &$\mathfrak{sp}(1)\oplus \mathfrak{sp}(n)$  &$ \mathfrak{sp}(n)$  & $n\geq 1$  \\
  \hline
  (2.1) & $\mathfrak{su}(2r+n)$   &$\mathfrak{su}(r)\oplus \mathfrak{su}(r+n)\oplus \mathbf{R} $  &$\mathfrak{su}(r)\oplus \mathfrak{su}(r+n)$  & $r\geq 2,n\geq 1$  \\
  \hline
  (2.2) & $\mathfrak{so}(4r+2)$   &$\mathfrak{u}(2r+1)$  &$\mathfrak{su}(2r+1)$  & $r\geq 2$  \\
  \hline
  (2.3) & $\mathfrak{e}(6)$   &$\mathfrak{so}(10)\oplus \mathbf{R}$  &$\mathfrak{so}(10)$  &   \\
  \hline
  (3.1) & $\mathfrak{so}(9)$   &$\mathfrak{so}(7)\oplus \mathfrak{so}(2)$  &$\mathfrak{g}_2\oplus \mathfrak{so}(2)$  &   \\
  \hline
  (3.2)& $\mathfrak{so}(10)$   &$\mathfrak{so}(8)\oplus \mathfrak{so}(2)$  &$\mathfrak{spin}(7)\oplus \mathfrak{so}(2)$  &   \\
  \hline
  (3.3) & $\mathfrak{so}(11)$   &$\mathfrak{so}(8)\oplus \mathfrak{so}(3)$  &$\mathfrak{spin}(7)\oplus \mathfrak{so}(3)$  &   \\
  \hline
\end{tabular}
\end{table}

By Theorem \ref{Theorem 3}, we see easily
\begin{corollary}\label{last-cor}
All standard $(\alpha_1,\alpha_2)$-metrics on $G/H$ associated with the triple $(G,K$, $H)$ in Table \ref{TableI} are $G$-g.o..
\end{corollary}

Among all the coset spaces $G/H$ associated with triples in Table \ref{TableI}, there are many weakly symmetric homogeneous spaces \cite{Ziller}, for example, the cases (1.1), (1.3), (1.4), (1.6), (1.7), (1.8) and (3.2). On the other hand, the cases (1.2), (1.9), (3.1) and (3.3) provide new examples of g.o. Finsler metrics on
 $SO(4n+1)/SU(2n)$, $Sp(n+1)/Sp(n)$, $SO(9)/(G_2\times SO(2))$ and $SO(11)/(Spin(7)\times SO(3))$ respectively,
 which are not weakly symmetric.

The case (1.5), (2.1), (2.2) and (2.3) are $\varphi$-symmetric spaces \cite{JK1993}\cite{Ta1977},
i.e., weakly symmetric with respect to the action of $G\times SO(2)$. By Lemma 4.2 in
\cite{Xu2018} and Corollary \ref{last-cor}, standard homogeneous $(\alpha_1,\alpha_2)$-metrics
on these coset spaces are $G\times SO(2)$-invariant. So they satisfy both the
$G\times SO(2)$-weakly symmetric property and the $G$-g.o. property.

For the special subcases $SO(5)/U(2)$, $Spin(8)/G_2$ and $Spin(9)/Spin(7)$ in Table \ref{TableI}, the generic orbits of isotropy $H$-actions
in $\mathfrak{m}=\mathfrak{m}_F+\mathfrak{m}_B$ is of codimension
two  \cite{Straume}. All $G$-invariant Finsler metrics on these coset spaces
are standard homogeneous $(\alpha_1,\alpha_2)$-metrics associated
to the triples.

Nextly, we consider standard homogeneous
$(\alpha_1,\alpha_2)$-metrics on the three Wallach
spaces \cite{W}
\begin{equation}
W^6=SU(3)/T^2,\, W^{12}=Sp(3)/Sp(1)^3\mbox{ and } W^{24}=F_4/Spin(8),\nonumber
\end{equation}
and prove

\begin{theorem}\label{last-thm}
Let $G/H$ be one of the Wallach spaces $W^6$, $W^{12}$, $W^{24}$, then any standard $G$-g.o. $(\alpha_1,\alpha_2)$-metric on $G/H$ is a $G$-normal homogeneous Riemannian metric.
\end{theorem}

\begin{proof} We
fix a bi-invariant inner product $|\cdot|_{\mathrm{bi}}^2=\langle\cdot,\cdot\rangle_{\mathrm{bi}}$ on $\mathfrak{g}$, which
is unique up to a scalar multiplication.
We have the $\mathrm{Ad}(H)$-invariant bi-invariant decomposition
$$\mathfrak{g}=\mathfrak{h}+\mathfrak{m}=
\mathfrak{h}+\mathfrak{m}_1+\mathfrak{m}_2+
\mathfrak{m}_3,$$
which satisfies
\begin{eqnarray*}
& &[\mathfrak{h},\mathfrak{m}_i]=\mathfrak{m}_i\mbox{ and } [\mathfrak{m}_i,\mathfrak{m}_i]\subset\mathfrak{h},
\mbox{ for each }i, \mbox{ and}\\
& &[\mathfrak{m}_i,\mathfrak{m}_j]\subset
\mathfrak{m}_k,\mbox{ when }
\{i,j,k\}=\{1,2,3\}.
\end{eqnarray*}
Furthermore, for any $i$, $(\mathfrak{g},\mathfrak{k}_i)$ is a rank-one compact symmetric pair, where $\mathfrak{k}_i=\mathfrak{h}+\mathfrak{m}_i$.

Since $\mathfrak{m}_1$, $\mathfrak{m}_2$ and $\mathfrak{m}_3$ are
distinct $\mathrm{Ad}(H)$-representations, any standard homogeneous $(\alpha_1,\alpha_2)$-metric $F$ on $G/H$ must be associated with the decomposition $\mathfrak{m}=\mathfrak{m}_{1,2}+\mathfrak{m}_{3}$, in which
$\mathfrak{m}_{1,2}=\mathfrak{m}_1+\mathfrak{m}_2$, up to a permutation for $\mathfrak{m}_i$'s. We assume the standard homogeneous $(\alpha_1,\alpha_2)$-metric is determined
by the norm $F(u)=|u|_{\mathrm{bi}}\varphi(|u_3|_{\mathrm{bi}}/
|u|_{\mathrm{bi}})$, where $u=u_1+u_2+u_3\neq0$ with
$u_i\in\mathfrak{m}_i$ for each $i$.

Denote $\tilde{\mathfrak{m}}=\mathfrak{m}_2+\mathfrak{m}_3$, and $\tilde{F}=F|_{\tilde{\mathfrak{m}}}$. Then $\tilde{F}$ is a Minkowski norm on $\tilde{\mathfrak{m}}$ which can be presented as
$\tilde{F}(u)=|u|_{\mathrm{bi}}\varphi(|u_3|_{\mathrm{bi}}
/|u|_{\mathrm{bi}})$,
where $\varphi$ is the same function in the presentation of $F$,
and $u=u_2+u_3\neq0$ with $u_2\in\mathfrak{m}_2$ and $u_3\in\mathfrak{m}_3$ respectively.

Let $K_1$ be the compact connected subgroup of $G$ generated by
$\mathfrak{k}_1=\mathfrak{h}+\mathfrak{m}_1$. Then we claim

{\bf Claim:} $\tilde{F}$ defines a homogeneous $(\alpha_1,\alpha_2)$-metric
on $G/K$.

To prove this claim, we consider any nonzero $\tilde{v}\in\tilde{\mathfrak{m}}\backslash\{0\}$, and any vector
$v_1\in\mathfrak{m}_1$.
By the $G$-g.o. assumption for $F$, we can find $u\in\mathfrak{h}$,
such that
\begin{eqnarray*}
g^F_{\tilde{v}}(\tilde{v},[\tilde{v}+u,v_1])=
g^F_{\tilde{v}}( \tilde{v},[\tilde{v}+u,v_1]_{\mathfrak{m}})=0.
\end{eqnarray*}

Direct calculation shows
$g^F_{\tilde{v}}(\tilde{v},\mathfrak{m}_1)=0$.
So $g^F_{\tilde{v}}(\tilde{v},[u,v_1])=0$, and thus
\begin{eqnarray}
g^{\tilde{F}}_{\tilde{v}}(\tilde{v},[v_1,\tilde{v}])
=g^{{F}}_{\tilde{v}}(\tilde{v},[v_1,\tilde{v}])
=g_{\tilde{v}}^F(\tilde{v},[u,v_1])=0,\label{0016}
\quad \forall v_1\in\mathfrak{m}_1.
\end{eqnarray}

Since the norm $F$ is $\mathrm{Ad}(H)$-invariant,
we have
\begin{eqnarray}\label{0017}
g_{\tilde{u}}^{\tilde{F}}(\tilde{v},[u',\tilde{v}])=0,
\quad\forall u'\in\mathfrak{h}.
\end{eqnarray}

Summarizing (\ref{0016}) and (\ref{0017}), we get
\begin{eqnarray*}
g^{\tilde{F}}_{\tilde{v}}(\tilde{v},
[\mathfrak{k}_1,\tilde{v}])=0,\,\forall \tilde{v}\in\tilde{\mathfrak{m}}\backslash\{0\},
\end{eqnarray*}
so $\tilde{F}$ is $\mathrm{Ad}(K_1)$-invariant, which
defines a homogeneous $(\alpha_1,\alpha_2)$-metric on $G/K_1$.

This argument proves our claim.

The compact coset space $G/K_1$ is one of following three
compact rank-one symmetric spaces, $\mathbb{C}\mathrm{P}^2=SU(3)/S(U(1)U(2))$,
$\mathbb{H}\mathrm{P}^2=Sp(3)/Sp(1)Sp(2)$ and
$\mathbb{O}\mathrm{P}^2=F_4/Spin(9)$. So the homogeneous
metric $\tilde{F}$ on $G/K_1$ must be a $G$-homogeneous Riemannian metric, i.e. the function $\varphi$ in its presenting is a constant function. Since the function $\varphi$ is also
used in the presentation $F$ as a standard homogeneous $(\alpha_1,\alpha_2)$-metric, $F$ must be a $G$-normal homogeneous metric as well.

This ends the proof of the theorem.
\end{proof}

Finally, we remark that
in \cite{AW} the authors prove that homogeneous Riemannian metrics
on generalized Wallach spaces are g.o. if and only if they are
normal homogeneous Riemannian metrics. Theorem \ref{last-thm}
implies their theorem might be generalized to Finsler geometry as well.

\end{document}